\def\year{2018}\relax

\documentclass[letterpaper]{article} 
\usepackage{aaai18}  

\usepackage{times}  
\usepackage{helvet}  
\usepackage{courier}  
\usepackage{url}  
\usepackage{graphicx}  
\usepackage{subfigure}
\usepackage[ruled,vlined]{algorithm2e}
\usepackage[usenames,dvipsnames,table]{xcolor}
\usepackage{latexsym,amssymb,amsmath,amsthm,mathrsfs}
\usepackage{bm}
\usepackage{hyperref}
\newcommand{\hide}[1]{}

\newcommand{\pl}[1]{\textcolor{ForestGreen}{[Palma: #1]}}

\newtheorem{theorem}{Theorem}[section]

\newtheorem{lemma}[theorem]{Lemma}

\newtheorem{corollary}[theorem]{Corollary}

\newtheorem{assumption}[theorem]{Assumption}
\newtheorem{claim}[theorem]{Claim}

  \newcommand{\eps}{\epsilon}
 
 \newcommand{\F}{\mathcal{F}}

  \newcommand{\A}{{\cal A}}

  \newcommand{\lp}{\mathcal{L}}

  \newcommand{\maximize}{\operatorname{maximize}}

 \newcommand{\expect}[1]{\mathbb{E}\left[#1\right] }
 \newcommand{\mbR}{\mathbb{R}}

\newcommand{\ignore}[1]{}

\newcommand{\es}{\eps_s} 
\newcommand{\ef}{\eps_f} 

\newcommand{\toteps}{\frac{(1-\ef) \es}{\dcs}}

\newcommand{\dcs}{\alpha_d}

\newcommand{\xsol}{\hat{x}}

\newcommand{\ysamp}{\tilde{y}}
\newcommand{\xsamp}{\tilde{x}}
\newcommand{\sphi}{\tilde{\phi}}
\newcommand{\spsi}{\tilde{\psi}}
\newcommand{\lpa}{$LP$~\eqref{LP1}}
\newcommand{\bnew}{\tilde{b}}

\usepackage{algorithm2e}
\usepackage{setspace}
\usepackage{enumitem}
\usepackage{booktabs}

\usepackage{footmisc}
\DefineFNsymbols{mySymbols}{{\ensuremath\dagger}{\ensuremath\ddagger}\S\P
   *{**}{\ensuremath{\dagger\dagger}}{\ensuremath{\ddagger\ddagger}}}
\setfnsymbol{mySymbols}

\author{}
\begin{document}
	\title{A Parallelizable Acceleration Framework for Packing Linear Programs}
        \author{Palma London  \\
        		California Institute of Technology \\
        		plondon@caltech.edu\\
        	\And Shai Vardi  \\
        		California Institute of Technology\\
        		svardi@caltech.edu\\
        \AND Adam Wierman\\
        		California Institute of Technology\\
        	adamw@caltech.edu
        \And Hanling Yi \\
        		The Chinese University of Hong Kong\\
        		yh014@ie.cuhk.edu.hk }



%
%

\maketitle
\begin{abstract}

This paper presents an acceleration framework for packing linear programming problems where the amount of data available is limited, i.e., where the number of constraints $m$ is small compared to the variable dimension $n$.  The framework can be used as a black box to speed up linear programming solvers dramatically, by two orders of magnitude in our experiments. We present worst-case guarantees on the quality of the solution and the speedup provided by the algorithm, showing that the framework provides an approximately optimal solution while running the original solver on a much smaller problem. The framework can be used to accelerate exact solvers, approximate solvers, and parallel/distributed solvers.  Further, it can be used for both linear programs and integer linear programs.  



\end{abstract}
\section{Introduction}

This paper proposes a black-box framework that can be used to accelerate  both exact and approximate linear programming (LP) solvers for packing problems while maintaining high quality solutions. 


LP solvers are at the core of many learning and inference problems, and often the linear programs of interest fall into the category of \emph{packing problems}.  Packing problems are linear programs of the following form:
\begin{subequations}\label{LP1}
\begin{align} 
    \maximize \quad & \textstyle \sum_{j=1} ^{n} c_j x_j &  \\
	\text{subject to} \quad & \textstyle \sum_{j=1} ^{n} a_{ij} x_j  \le b_i & i \in [m]  \\
	& 0 \leq x_j \leq 1  & j \in [n] \label{const1}
\end{align}
\end{subequations}
where  $A \in [0,1]^{m \times n}$, $b \in \mathbb{R}_{\geq 0}^{m}$, $c \in \mathbb{R}_{\geq 0}^{n}$.


Packing problems arise in a wide variety of settings, including max cut \cite{Trevisan98}, zero-sum matrix games \cite{Nesterov05}, scheduling and graph embedding \cite{Tardos95}, flow controls \cite{bbr04}, auction mechanisms \cite{zn01}, wireless sensor networks \cite{bn00}, and many other areas.  In machine learning specifically, they show up in an array of problems, e.g., in applications of graphical models \cite{raw10}, associative Markov networks \cite{Taskar2004}, and in relaxations of maximum a posteriori (MAP) estimation problems \cite{NIPS_smw08}, among others.

In all these settings, practical applications require LP solvers to work at extreme scales and, despite decades of work, commercial solvers such as Cplex and Gurobi do not scale as desired in many cases.  Thus, despite a  large literature, the development of fast, parallelizable solvers for packing LPs is still an active direction.  

Our focus in this paper is on a specific class of packing LPs for which data is either very costly, or hard to obtain.  In these situations $m \ll n$; i.e., the number of data points $m$ available is much smaller than the number of variables, $n$.  Such instances are common in areas such as genetics, astronomy, and chemistry. There has been considerable research focusing on this class of problems in recent years, in the context of LPs \cite{Donoho05,Bienstock06} and also more generally in convex optimization and compressed sensing \cite{Candes06,Donoho06compressedsensing}, low rank matrix recovery \cite{RechtFazelParrilo06,candes2011tight}, and graphical models \cite{Yuan07,Mohan_etal15}. 

\subsubsection*{Contributions of this paper.}
We present a black-box acceleration framework for LP solvers. 
When given a packing LP and an algorithm $\A$, the framework works by sampling an $\es$-fraction  of the variables and using $\A$ to solve $\lpa$ restricted to these variables. Then, the dual solution to this sampled LP is used to define a thresholding rule for the primal variables of the original LP;  the variables are set 	to either $0$ or $1$ according to this rule. The framework has the following key properties:
\begin{enumerate}
\item It can be used to accelerate exact or approximate LP-solvers (subject to some mild assumptions which we discuss below).
\item Since the original algorithm $\A$ is run only on a (much smaller) LP with $\es$-fraction of the variables, the framework provides a dramatic speedup.
\item The threshold rule can be used to set the values of the variables in parallel. Therefore, if $\A$ is a parallel algorithm, the framework gives a faster parallel algorithm with negligible overhead.
\item Since the threshold rule  sets the variables to integral values, the framework can be applied without modification to solve integer programs that have the same structure as $\lpa$, but with integer constraints replacing~\eqref{const1}.
\end{enumerate}


There are two fundamental tradeoffs in the framework.  The first is captured by the sample size, $\es$.  Setting $\es$ to be small yields a dramatic speedup of the algorithm $\A$; however, if $\es$ is set too small the quality of the solution suffers.  A second tradeoff involves feasibility.  In order to ensure that the output of the framework is feasible w.h.p.\ (and not just that each constraint is satisfied in expectation), the constraints of the sample LP are scaled down by a factor denoted by $\ef$. Feasibility is guaranteed if $\ef$ is large enough; however, if it is too large, the quality of the solution (as measured by the approximation ratio) suffers.   

Our main technical result is a worst-case characterization of the impact of $\es$ and $\ef$ on the speedup provided by the framework and the quality of the solution.  Assuming that algorithm $\A$ gives a $(1+\delta)$ approximation to the optimal solution of the dual, we prove that the acceleration framework guarantees a 
$(1-\ef)/(1+\delta)^2$-approximation to the optimal solution of $\lpa$, under some assumptions about the input and $\ef$. We formally state the result as Theorem \ref{t.main}, and note here that the result shows that $\ef$ grows proportionally to $1/\sqrt{\es}$, which highlights that the framework maintains a high-quality approximation even when sample size is small (and thus the speedup provided by the framework is large). 

The technical requirements for $\ef$ in Theorem \ref{t.main} impose some restrictions on both the family of LPs that can be provably solved using our framework and the algorithms that can be accelerated. In particular, Theorem \ref{t.main} requires $\min_i {b_i}$ to be large and the algorithm $\A$ to satisfy approximate complementary slackness conditions (see Section \ref{s.alg}).  While the condition on the $b_i$ is restrictive, the condition on the algorithms is not -- it is satisfied by most common LP solvers, e.g., exact solvers and  many primal dual approximation algorithms.  Further, our experimental results demonstrate that these technical requirements are conservative -- the framework produces solutions of comparable quality to the original LP-solver in settings that are far from satisfying the theoretical requirements. In addition, the accelerator works in practice for algorithms that do not satisfy approximate complementary slackness conditions, e.g., for gradient algorithms as in \cite{S13}. In particular, our experimental results show that the accelerator obtains solutions that are close in quality to those obtained by the algorithms being accelerated on the complete problem, \emph{and} that the solutions are obtained considerably faster (by up to two orders of magnitude).  The results reported in this paper demonstrate this by accelerating the state-of-the-art commercial solver Gurobi on a wide array of randomly generated packing LPs and obtaining solutions with $<4\%$ relative error and a more than $150\times$ speedup. Other experiments with other solvers are qualitatively similar and are not included.   

When applied to parallel algorithms, there are added opportunities for the framework to reduce error while increasing the speedup, through \emph{speculative execution}: the framework runs multiple \emph{clones} of the algorithm speculatively. The original algorithm is executed on a separate sample and the thresholding rule is then applied by each clone in parallel, asynchronously.  This improves both the solution quality and the speed. It improves the quality of the solution because the best solution across the multiple samples can be chosen.  It improves the speed because it mitigates the impact of \emph{stragglers}, tasks that take much longer than expected due to contention or other issues.  Incorporating ``cloning'' into the acceleration framework triples the speedup obtained, while reducing the error by $12\%$.

\subsubsection*{Summary of related literature.}

The approach underlying our framework is motivated by recent work that uses ideas from online algorithms to make offline algorithms more scalable, e.g., \cite{MansourRVX12,LOCO}. A specific inspiration for this work is \cite{Agrawal}, which introduces an online algorithm that uses a two step procedure: it solves an LP based on the first $s$ stages and then uses the solution as the basis of a rounding scheme in later stages. The algorithm only works when the arrival order is random, which is analogous to sampling in the offline setting. However, \cite{Agrawal} relies on exactly solving the LP given by the first $s$ stages; considering approximate solutions of the sampled problem (as we do) adds complexity to the algorithm and analysis. Additionally, we can leverage the offline setting to fine-tune $\ef$ in order to optimize our solution while ensuring feasibility. 

The sampling phase of our framework is reminiscent of the method of \textit{sketching} in which the data matrix is multiplied by a random matrix in order to compress the problem and thus reduce computation time by working on a smaller formulation, e.g., see \cite{woodruff14}. 
However, sketching is designed for overdetermined linear regression problems, $m \gg n$; thus compression is desirable. In our case, we are concerned with underdetermined problems, $m \ll n$; thus compression is not appropriate. 
 Rather, the goal of sampling the variables is to be able to approximately determine the thresholds in the second step of the framework.  This difference means the approaches are distinct.

The sampling phase of the framework is also reminiscent of the \textit{experiment design} problem, in which the goal is to solve the least squares problem using only a subset of available data while minimizing the error covariance of the estimated parameters, see e.g.,  \cite{Boydbook}. Recent work \cite{Baosen2016} applies these ideas to online algorithms, when collecting data for regression modeling. Like sketching, experiment design is applied in the overdetermined setting, whereas we consider the under-determined scenario. Additionally, instead of sampling constraints, we sample variables. 


The second stage of our algorithm is a thresholding step and is related to the rich literature of LP rounding, see \cite{bv94} for a survey.  Typically, rounding is used to arrive at a solution to an ILP; however we use thresholding to ``extend'' the solution of a sampled LP to the full LP. The scheme we use is a deterministic threshold based on the complementary slackness condition.  It is inspired by \cite{Agrawal}, but adapted to hold for approximate solvers rather than exact solvers.  In this sense, the most related recent work is \cite{S13}, which proposes a scheme for rounding an approximate LP solution. However, \cite{S13} uses all of the problem data during the approximation step, whereas we show that it is enough to use a (small) sample of the data.  

A key feature of our framework is that it can be parallelized easily when used to accelerate a distributed or parallel algorithm.  There is a rich literature on distributed and parallel LP solvers, e.g., \cite{ys09,nb11,bgfa12,rc15}. More specifically, there is significant interest in distributed strategies for approximately solving covering and packing linear problems, such as the problems we consider here, e.g., \cite{Luby,young01,bbr04,ak08,AllenOrecchia2015}.

\section{A Black-Box Acceleration Framework}
\label{s.alg}

In this section we formally introduce our acceleration framework.  At a high level, the framework accelerates an LP solver by running the solver in a black-box manner on a small sample of variables and then using a  deterministic thresholding scheme to set the variables in the original LP. 
The framework can be used to accelerate any LP solver that satisfies the approximate complementary slackness conditions. The solution of an approximation algorithm $\A$ for a family of linear programs $\F$ satisfies the approximate complementary slackness if the following holds.  Let $x_1, \ldots, x_n$ be a feasible solution to the primal and $y_1, \ldots, y_m$ be a feasible solution to the dual.
\begin{itemize}
	\item \textit{Primal Approximate Complementary Slackness:} For $\alpha_p\geq	1$ and $j \in [n]$, if 
	$x_j>0$ then $c_j \leq \sum_{i=1}^m a_{ij}y_i \leq \alpha_p \cdot c_j.$
	\item \textit{Dual Approximate Complementary Slackness:} For $\alpha_d\geq 1$ and $i \in [m]$, 
	if $y_i>0$ then $b_i/\alpha_d \leq \sum_{i=1}^n a_{ij}x_j \leq b_i$.
\end{itemize}
We call an algorithm $\A$  whose solution is guaranteed to  satisfy the above conditions  an \emph{$(\alpha_p,\alpha_d)$-approximation algorithm} for $\F$. This terminology is non-standard, but is instructive when describing our results. It stems from a foundational result which states that an algorithm $\A$ that satisfies the above conditions is an $\alpha$-approximation algorithm for any LP in $\F$ for $\alpha=\alpha_p \alpha_d$ ~\cite{nivbook}. 

The framework we present can be used to accelerate any  $(1,\alpha_d)$-approximation algorithm.  While this is a stronger condition than simply requiring that $\A$ is an $\alpha$-approximation algorithm, many common dual ascent  algorithms satisfy this condition, e.g.,~\cite{ajit,won,bye,erl,gw}.  For example, the  vertex cover and Steiner tree approximation algorithms of ~\cite{ajit} and~\cite{bye} respectively are both $(1,2)$-approximation algorithms.

Given a $(1,\alpha_d)$-approximation algorithm $\A$, the acceleration framework  works in two steps.  The first step is to sample a subset of the variables, $S \subset [n]$, $|S|=s=\lceil  \es n \rceil$, and  use $\A$ to solve the following \emph{sample LP}, which we call LP~\eqref{sampleLP2}. For clarity, we relabel the variables so that the  sampled variables are labeled $1,\ldots,s$.
\begin{subequations}\label{sampleLP2}
\begin{align}
\maximize \quad &  \textstyle \sum_{j=1} ^{s} c_j x_j  \\
\text{subject to} \quad &  \textstyle \sum_{j=1} ^{s} a_{ij} x_j  \le \toteps b_i & i\in[m]  \label{line2}\\
& x_j \in [0,1]  & j \in[s] \label{line3}
\end{align}
\end{subequations}
Here, $\alpha_d$ is the parameter of the dual approximate complementary slackness guarantee of $\A$, $\ef>0$ is a parameter set to ensure feasibility during the thresholding step, and $\es>0$ is a parameter that determines the fraction of the primal variables that are be sampled.  Our analytic results give insight for setting $\ef$ and $\es$ but, for now, both should be thought of as close to zero.  Similarly, while the results hold for any $\alpha_d$, they are most interesting when $\alpha_d$ is close to $1$ (i.e., $\alpha_d = 1+\delta$ for small $\delta$). There are many such algorithms, given the recent interest in designing approximation algorithms for LPs, e.g., \cite{S13,AllenOrecchia2015}. 

The second step in our acceleration framework uses the dual prices from the sample LP in order to set a threshold for a deterministic thresholding procedure, which is used to build the solution of $\lpa$.  Specifically, let $\phi \in \mbR^{m}$ and $\psi \in \mbR^{s}$ denote the dual variables corresponding to the constraints~\eqref{line2} and ~\eqref{line3} in the sample LP, respectively. We define the allocation (thresholding) rule $x_j(\phi)$ as follows:
\begin{equation*} \label{alloc_rule_2}
\begin{aligned}
x_j(\phi) =
\left\{
\begin{array}{cl}
1  & \text{if }  \sum_{i=1}^{m} a_{ij} \phi_i <  c_j \\
0  & \text{otherwise }  \\
\end{array}
\right.
\end{aligned}
\end{equation*}

\begin{algorithm}[t]
	\DontPrintSemicolon  
	\KwIn{
		Packing LP $\lp$, LP solver $\A$, $\es>0$, $\ef>0$}
	\KwOut{$\hat{x} \in \mbR^{n}$}
	%
	\begin{spacing}{1.5}
	\end{spacing}
	\begin{enumerate}
		\item Select  $s=\lceil n  \es \rceil$ primal variables uniformly at random. Label this set $S$.
		\item Use $\A$ to find an (approximate) dual  solution $\ysamp=[\phi, \psi] \in [\mbR^{m},\mbR^{s}]$ to  the sample LP. 
		\item Set $\xsol_j = x_j(\phi)$ for all $j \in [n]$.
		\item Return $\xsol$.
			\end{enumerate}
	\caption{Core acceleration algorithm}
	\label{alg1}
\end{algorithm}

We summarize the core algorithm of the acceleration framework described above in Algorithm \ref{alg1}.  When implementing the acceleration framework it is desirable to search for the minimal $\ef$ that allows for feasibility.  This additional step is included in the full pseudocode of the acceleration framework given in Algorithm \ref{alg2}.  

\begin{algorithm}[t]
	\DontPrintSemicolon  
	\KwIn{
		Packing LP $\lp$, LP solver $\A$, $\es>0$, $\ef>0$}
	\KwOut{$\hat{x} \in \mbR^{n}$}
	%
	\begin{spacing}{1.5}
	\end{spacing}
	Set $\ef=0$.

\While{$\epsilon_f<1$}{
$\hat{x}$= Algorithm~\ref{alg1}($\lp$,$\A$,$\es$,$\ef$).\;
\If {$\hat{x}$ is a feasible solution to $\lp$}{Return $\hat{x}$.} 
\Else{Increase $\ef$.\;}
}
	\caption{Pseudocode for the full framework. }
	\label{alg2}
\end{algorithm}

It is useful to make a few remarks about the generality of this framework. First, since the allocation rule functions as a thresholding rule, the final solution output by the accelerator is integral.  Thus, it can be viewed as an ILP solver based on relaxing the ILP to an LP, solving the LP, and rounding the result.  The difference is that it does not solve the full LP, but only a (much smaller) sample LP; so it provides a significant speedup over traditional approaches.  
Second, the framework is easily parallelizable.  The thresholding step can be done independently and asynchronously for each variable and, further, the framework can easily integrate speculative execution.  Specifically, the framework can start multiple \emph{clones} speculatively, i.e., take multiple samples of variables, run the algorithm $\A$ on each sample, and then round each sample in parallel.  This provides two benefits.  First, it improves the quality of the solution because the output of the ``clone'' with the best solution can be chosen.  Second, it improves the running time since it curbs the impact of \textit{stragglers}, tasks that take much longer than expected due to contention or other issues.  Stragglers are a significant source of slowdown in clusters, e.g., nearly one-fifth of all tasks can be categorized as stragglers in Facebook's Hadoop cluster \cite{Dolly13}.  There has been considerable work designing systems that reduce the impact of stragglers, and these primarily rely on speculative execution, i.e., running multiple clones of tasks and choosing the first to complete  \cite{ananthanarayanan2010reining,Grass14,Ren_Hopper15}.  Running multiple clones in our acceleration framework has the same benefit.  To achieve both the improvement in solution quality and running time, the framework runs $K$ clones in parallel and chooses the best solution of the first $k<K$ to complete. We illustrate the benefit of this approach in our experimental results in Section \ref{sec:res}.

\section{Results}
\label{sec:res}

In this section we present our main technical result, a worst-case characterization of the quality of the solution provided by our acceleration framework. We then illustrate the speedup provided by the framework through experiments using Gurobi, a state-of-the-art commercial solver. 

\subsection{A Worst-case Performance Bound}

The following theorem bounds the quality of the solution provided by the acceleration framework. 
Let $\lp$ be a packing LP with $n$ variables and $m$ constraints, as in \eqref{LP1}, and $B := \min_{i \in [m]} \{b_i\}$. For simplicity, take $\es n$ to be integral.

\begin{theorem} \label{t.main} 
Let  $\A$ be an $(1,\alpha_d)$-approximation algorithm for packing LPs, with runtime $f(n,m)$. For any $\es>0$ and $\ef \geq 3\sqrt{\frac{6(m+2)\log{n}}{\es B}}$, Algorithm \ref{alg1} runs in time $f(\es n, m)+O(n)$ and obtains a feasible $(1-\ef)/\alpha_d^2$-approximation to the optimal solution for $\lp$ with probability at least $1-\frac{1}{n}$. 
\end{theorem}
\begin{proof}
The approximation ratio follows from Lemmas~\ref{lem:shipra1} and~\ref{lem:shipra2} in Section \ref{s.proof}, with a rescaling of $\ef$ by 1/3 in order to simplify the theorem statement. The runtime follows from the fact that $\A$ is executed on an LP with $\es n$ variables and at most $m$ constraints and that, after running $\A$, the thresholding step is used to set the value for all $n$ variables. 
\end{proof}

The key trade-off in the acceleration framework is between the size of the sample LP, determined by $\es$, and the resulting quality of the solution, determined by the feasibility parameter, $\ef$.  The accelerator provides a large speedup if $\es$ can be made small without causing $\ef$ to be too large.  Theorem \ref{t.main} quantifies this trade-off: $\ef$ grows as $1/\sqrt{\es}$.  Thus, $\es$ can be kept small without impacting  the loss in solution quality too much.  The bound on $\ef$ in the theorem also defines the class of problems for which the accelerator is guaranteed to perform well---problems where $m \ll n$ and $B$ is not too small.  
Nevertheless, our experimental results successfully apply the framework well outside of these parameters---the theoretical analysis provides a very conservative view on the applicability of the framework. 


Theorem \ref{t.main} considers the acceleration of $(1,\alpha_d)$-approximation algorithms.  As we have already noted, many common approximation algorithms fall into this class. Further, any exact solver satisfies this condition.  For exact solvers, Theorem \ref{t.main} guarantees a $(1-\ef)$-approximation (since $\alpha_d=1$). 

In addition to exact and approximate LP solvers, our framework can also be used to convert LP solvers into ILP solvers, since the solutions it provides are always integral; and it can be parallelized easily, since the thresholding step can be done in parallel.  We emphasize these points below.

\begin{corollary}
Let  $\A$ be an $(1,\alpha_d)$-approximation algorithm for packing LPs, with runtime $f(n,m)$. Consider $\es>0$, and $\ef \geq 3\sqrt{\frac{6(m+2)\log{n}}{\es B}}$.
\begin{itemize}
\item Let $\mathcal{IL}$ be an integer program similar to LP~\eqref{LP1} but with integrality constraints on the variables. Running Algorithm \ref{alg1} on LP~\eqref{LP1} obtains a feasible $(1-\ef)/\alpha_d^2$-approximation to the optimal solution for $\mathcal{IL}$ with probability at least $1-\frac{1}{n}$ with runtime $f(\es n,m) +O(n)$.
\item If $\A$ is a parallel algorithm, then executing Algorithm~\ref{alg1} on $p$ processors in parallel obtains a feasible $(1-\ef)/\alpha_d^2$-approximation to the optimal solution for $\lp$ or $\mathcal{IL}$ with probability at least $1-\frac{1}{n}$ and runtime $f_p(\es n,m) +O(n/p)$, where $f_p(\es n, m)$ denotes $\A$'s runtime for the sample program on $p$ processors.
\end{itemize}
\end{corollary}

\subsection{Accelerating Gurobi}

We illustrate the speedup provided by our acceleration framework by using it to accelerate Gurobi, a state-of-the-art commercial solver.  Due to limited space, we do not present results applying the accelerator to other, more specialized, LP solvers; however the improvements shown here provide a conservative estimate of the improvements using parallel implementations since the thresholding phase of the framework has a linear speedup when parallelized.  Similarly, the speedup provided by an exact solver (such as Gurobi) provides a conservative estimate of the improvements when applied to  approximate solvers or when applied to solve ILPs.

Note that our experiments consider situations where the assumptions of Theorem \ref{t.main} about $B$, $m$, and $n$ do not hold.  Thus, they highlight that the assumptions of the theorem are conservative and the accelerator can perform well outside of the settings prescribed by Theorem \ref{t.main}.  This is also true with respect to the assumptions on the algorithm being accelerated.  While our proof requires the algorithm to be a $(1,\alpha_d)$-approximation, the accelerator works well for other types of algorithms too.  For example, we have applied it to gradient algorithm such as \cite{S13} with results that parallel those presented for Gurobi below.

\subsubsection{Experimental Setup.} To illustrate the performance of our accelerator, we run 
Algorithm \ref{alg2}
on randomly generated LPs.  Unless otherwise specified, the experiments use a matrix $A\in \mathbb{R}^{m\times n}$ of size $m=10^2, n=10^6$. Each element of $A$, denoted as $a_{ij}$, is first generated from $[0,1]$ uniformly at random and then set to zero with probability $1-p$. Hence, $p$ controls the sparsity of matrix $A$, and we vary $p$ in the experiments. The vector $c\in \mathbb{R}^n$ is drawn i.i.d.\ from $[1,100]$ uniformly. Each element of the vector $b\in \mathbb{R}^m$ is fixed as $0.1n$. (Note that the results are qualitatively the same for other choices of $b$.) By default, the parameters of the accelerator are set as $\epsilon_s=0.01$ and $\epsilon_f=0$, though these are varied in some experiments. Each point in the presented figures is the average of over 100 executions under different realizations of $A, c$.

To assess the quality of the solution, we measure the \textit{relative error} and \textit{speedup} of the accelerated algorithm as compared to the original algorithm. The relative error is defined as $(1-Obj/OPT)$, where $Obj$ is the objective value produced by our algorithm and $OPT$ is the optimal objective value. The speedup is defined as the run time of the original LP solver divided by that of our algorithm. 

We implement the accelerator in Matlab and use it to accelerate Gurobi. The experiments are run on a server with Intel E5-2623V3@3.0GHz 8 cores and 64GB RAM. We intentionally perform the experiments with a small degree of parallelism in order to obtain a conservative estimate of the acceleration provided by our framework.  As the degree of parallelism increases, the speedup of the accelerator increases and the quality of the solution remains unchanged (unless cloning is used, in which case it improves).  

\subsubsection{Experimental Results.} 
Our experimental results highlight that our acceleration framework provides speedups of two orders of magnitude (over $150\times$), while maintaining high-quality solutions (relative errors of $<4$\%).

\emph{The trade-off between relative error and speed.} The fundamental trade-off in the design of the accelerator is between the sample size, $\es$, and the quality of the solution.  The speedup of the framework comes from choosing $\es$ small, but if it is chosen too small then the quality of the solution suffers.  For the algorithm to provide improvements in practice, it is important for there to be a sweet spot where $\es$ is small and the quality of the solution is still good, as indicated in the shaded region of Figure \ref{fig:epsilon}.

\begin{figure}[t]
  \centering
  \subfigure[$p=0.8$]{
    \label{fig:epsilon_a} 
    \includegraphics[width=1.6in]{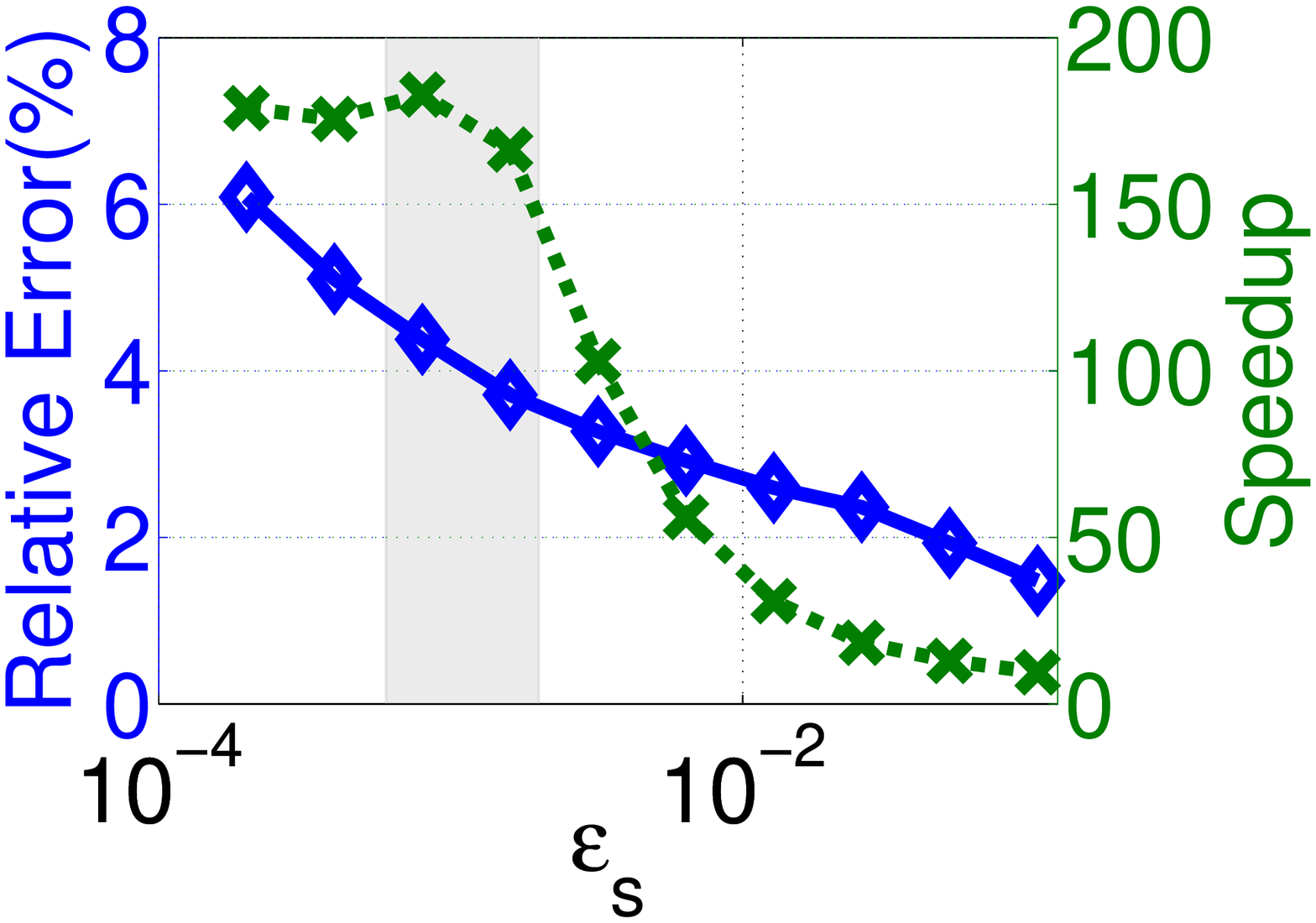}}
  \subfigure[$p=0.4$]{
    \label{fig:epsilon_b} 
    \includegraphics[width=1.6in]{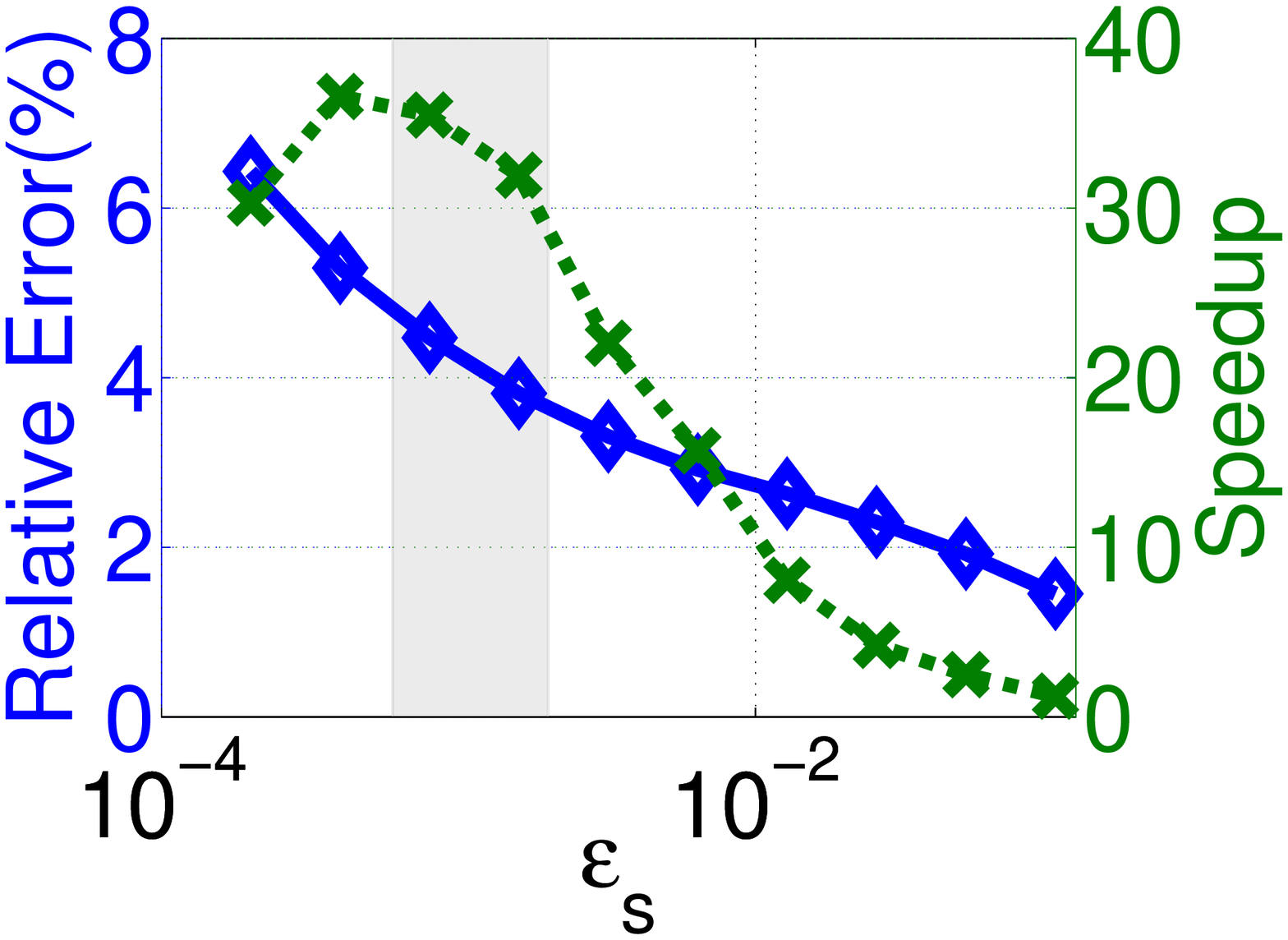}}
  \caption{\emph{Illustration of the relative error and speedup across sample sizes, $\es$. Two levels of sparsity, $p$, are shown.}}
\label{fig:epsilon} 
\end{figure}

\emph{Scalability.}  In addition to speeding up LP solvers, our acceleration framework provides significantly improved scalability.  Because the LP solver only needs to be run on a (small) sample LP, rather than the full LP, the accelerator provides order of magnitude increase in the size of problems that can be solved.  This is illustrated in Figure~\ref{fig:problem_size}.  The figure shows the runtime and relative error of the accelerator.  In these experiments we have fixed $p=0.8$ and $n/m=10^3$ as we scale $m$. We have set  $\es=0.01$ throughout.  As (a) shows, one can choose $\es$ more aggressively in large problems since leaving $\es$ fixed leads to improved accuracy for large scale problems. Doing this would lead to larger speedups; thus by keeping $\es$ fixed we provide a conservative estimate of the improved scalability provided by the accelerator.  The results in (b) illustrate the improvements in scalability provided by the accelerator.  Gurobi's run time grows quickly until finally, it runs into memory errors and cannot arrive at a solution.  In contrast, the runtime of the accelerator grows slowly and can (approximately) solve problems of much larger size.  To emphasize the improvement in scalability, we run an experiment on a laptop with Intel Core i5 CPU and 8 GB RAM. For a problem with size $m=10^2, n=10^7$, Gurobi fails due to memory limits. In contrast, the accelerator produces a solution in $10$ minutes with relative error less than $4\%$.

\begin{figure}[t]
  \centering
\subfigure[Relative Error]{
    \label{fig:problem_size_a} 
    \includegraphics[width=1.6in]{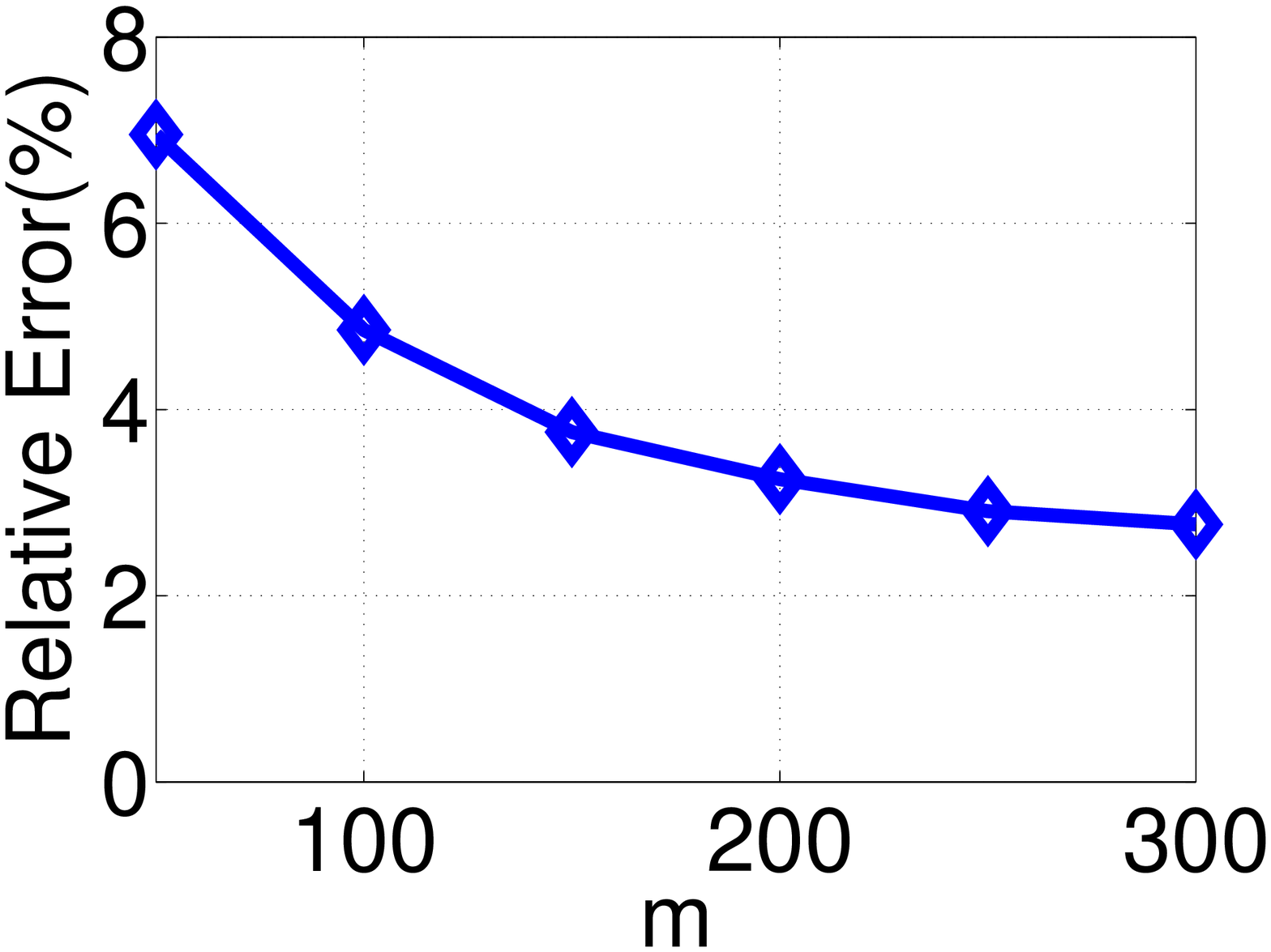}}
\subfigure[Runtime]{
    \label{fig:problem_size_b} 
    \includegraphics[width=1.6in]{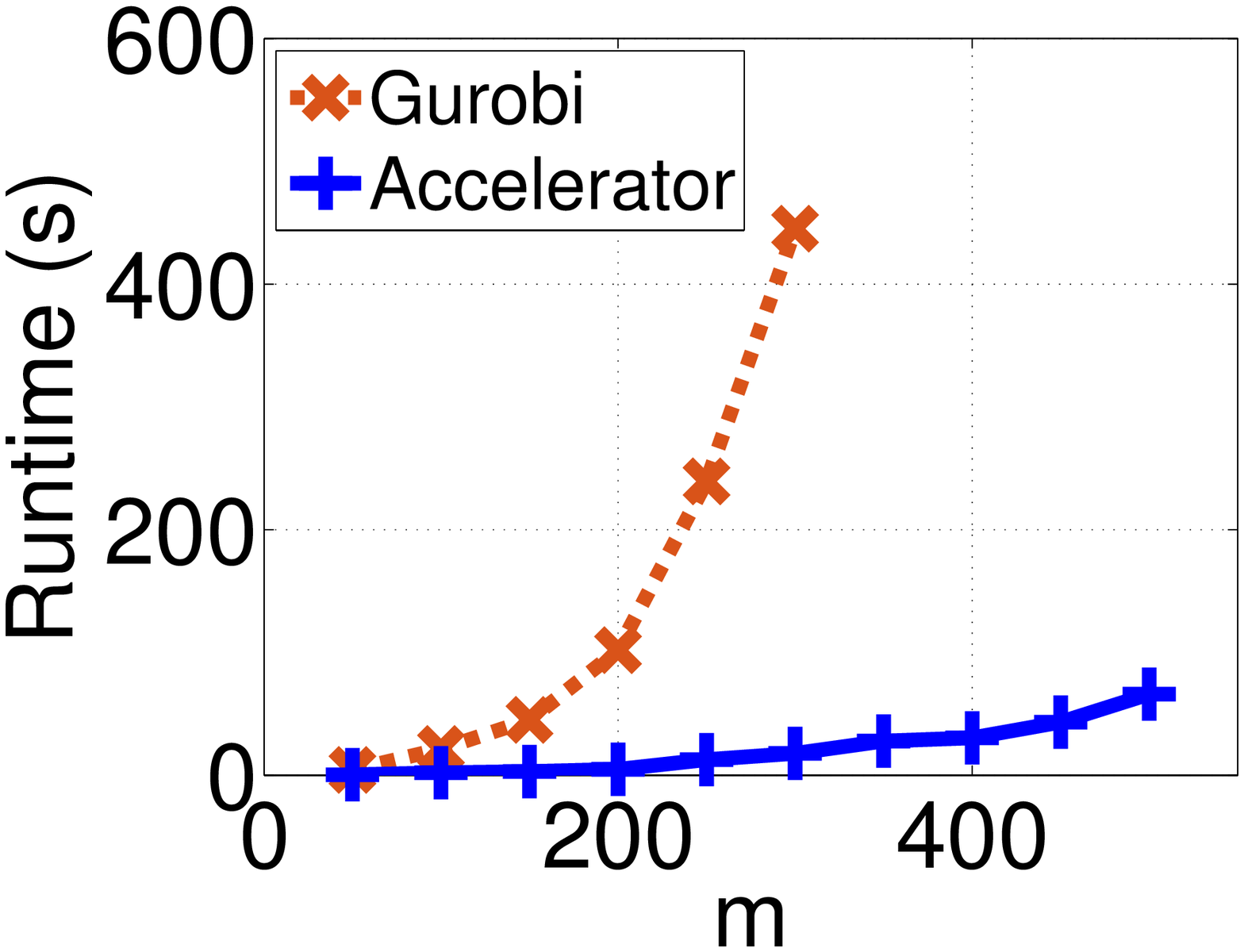}}
\caption{\emph{Illustration of the relative error and runtime as the problem size, $m$, grows.}}
  \label{fig:problem_size} 
\end{figure}

\emph{The benefits of cloning.} Speculative execution is an important tool that parallel analytics frameworks use to combat the impact of stragglers.  Our acceleration framework can implement speculative execution seamlessly by running multiple clones (samples) in parallel and choosing the ones that finish the quickest.  We illustrate the benefits associated with cloning in Figure~\ref{Figure_eps_boost}.  This figure shows the percentage gain in relative error and speedup associated with using different numbers of clones. In these experiments, we fix $\epsilon=0.002$ and $p=0.8$.  We vary the number of clones run and the accelerator outputs a solution after the fastest four clones have finished.  Note that the first four clones do not impact the speedup as long as they can be run in parallel.  However, for larger numbers of clones our experiments provide a conservative estimate of the value of cloning since our server only has 8 cores.  The improvements would be larger than shown in Figure~\ref{Figure_eps_boost} in a system with more parallelism. Despite this conservative comparison, the improvements illustrated in Figure~\ref{Figure_eps_boost} are dramatic.  Cloning reduces the relative error of the solution by $12\%$ and triples the speedup.  Note that these improvements are significant even though the solver we are accelerating is not a parallel solver.  

\begin{figure}[t]
    \centering
    \subfigure[Relative Error]{
    \includegraphics[width=1.6in]{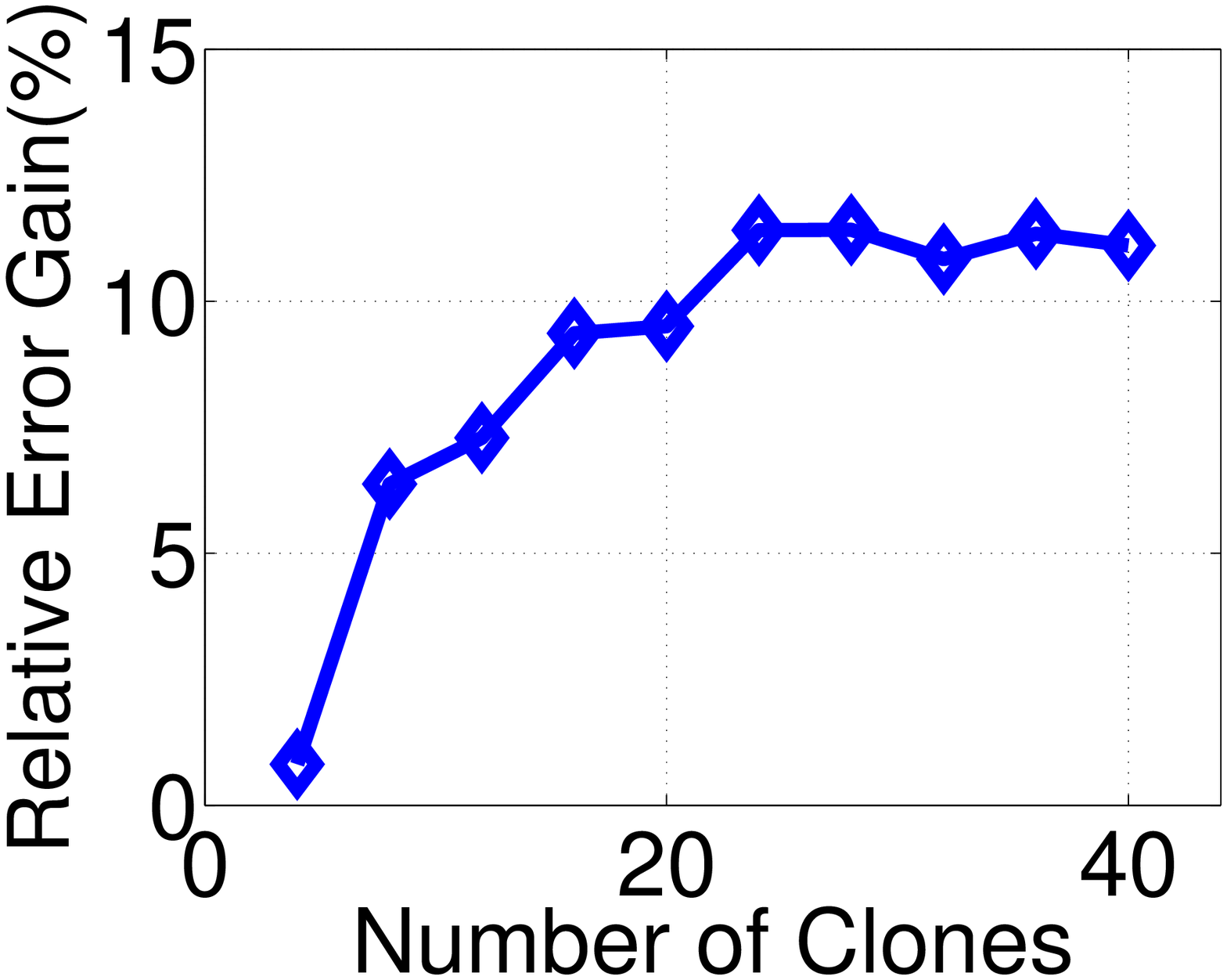}}
    \subfigure[Speedup]{
    \includegraphics[width=1.6in]{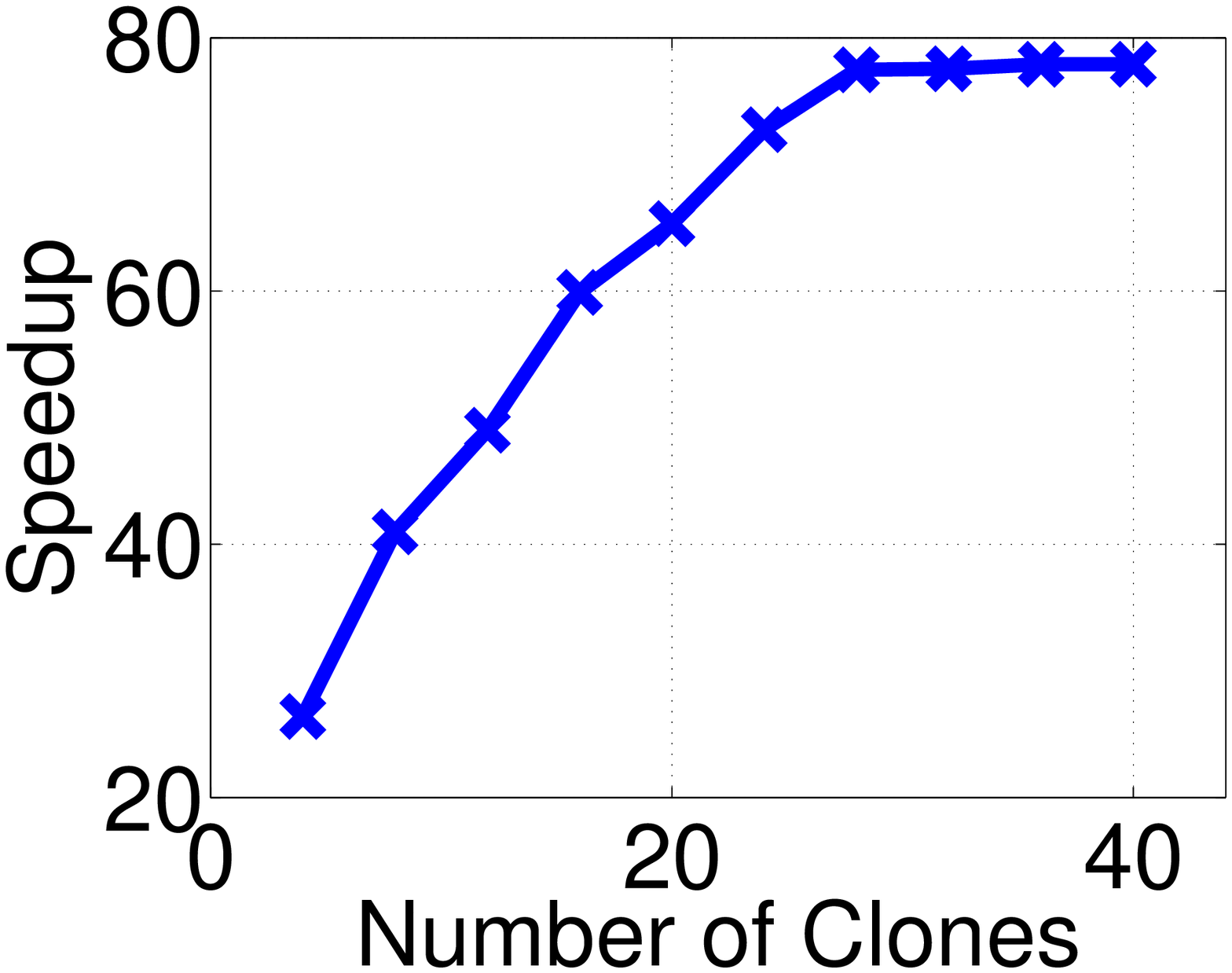}}
    \caption{\emph{Illustration of the impact of cloning on solution quality as the number of clones grows.}}
\label{Figure_eps_boost}
 \end{figure}

\subsection{Case Study}


To illustrate the performance in a specific practical setting, we consider an example focused on optimal resource allocation in a network.  We consider an LP that represents a multi-constraint knapsack problem associated with placing resources at intersections in a city transportation network. For example, we can place public restrooms, advertisements, or emergency supplies at intersections in order to maximize social welfare, but such that there never is a particularly high concentration of resources in any area. 

Specifically, we consider a  subset of the California road network dataset \cite{snap}, consisting of $100,000$ connected traffic intersections. 
We consider only a subset of a total of $1,965,206$ intersections because Gurobi is unable to handle such a large dataset when run on a laptop with Intel Core i5 CPU and 8 GB RAM. 
We choose $1000$ of the $100,000$ intersections uniformly at random and defined for each of them a local vicinity of $20,000$ neighboring intersections, allowing overlap between the vicinities. 
The goal is to place resources strategically at intersections, such that the allocation is not too dense within each local vicinity. 
Each intersection is associated with a binary variable which represents a yes or no decision to place resources there. Resources are constrained such that the sum of the number of resource units placed in each local vicinity does not exceed $10,000$. 

Thus, the dataset is as follows. Each element $A_{ij}$ in the data matrix is a binary value representing whether or not the $i$-th intersection is part of the $j$-th  local vicinity. There are $1000$  local vicinities and $100,000$ intersections, hence $A$ is a $(1000 \times 100,000)$ matrix.
Within each local vicinity, there are no more than $b_j = 10,000$ resource units. 

The placement of resources at particular locations has an associated utility, which is a quantifier of how beneficial it is to place resources at various locations. For example, the benefit of placing public restrooms, advertisements, or emergency supplies at certain locations may be proportional to the population of the surrounding area. In this problem, we randomly draw the utilities from Unif$[1,10]$. The objective value is the sum of the utilities at whose associated nodes resources are placed.

Figure \ref{Figure_RD} demonstrates the relative error and runtime of the accelerator compared to Gurobi, as we vary the sample size $\epsilon_s$. There is a speed up by a factor of more than $30$ when the approximation ratio is $0.9$, or a speed up by a factor of about $9$ when the approximation ratio is $0.95$.

\begin{figure}[t]
    \centering
    \subfigure[Relative Error]{
    \includegraphics[width=1.6in]{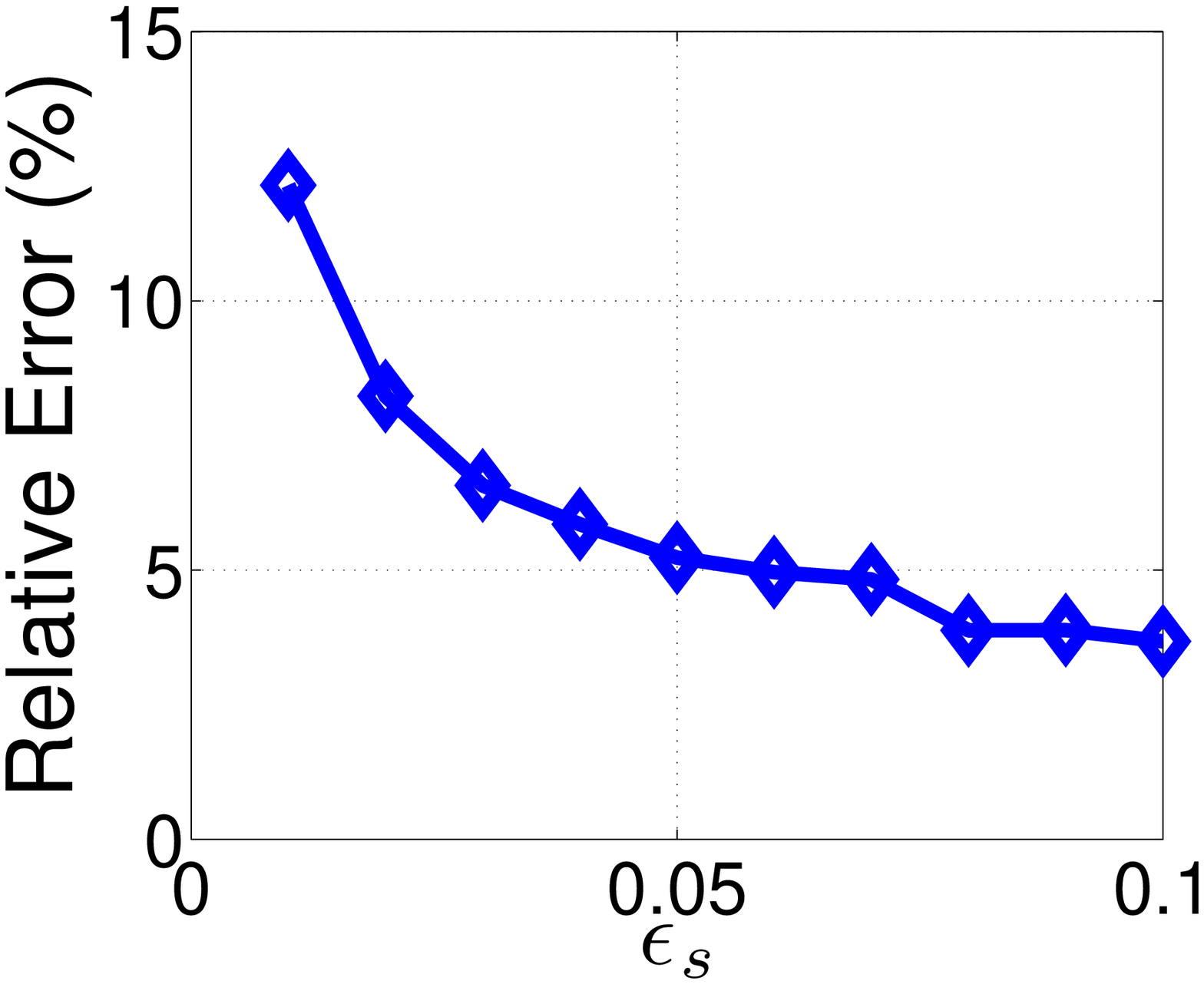}}
    \subfigure[Runtime]{
    \includegraphics[width=1.6in]{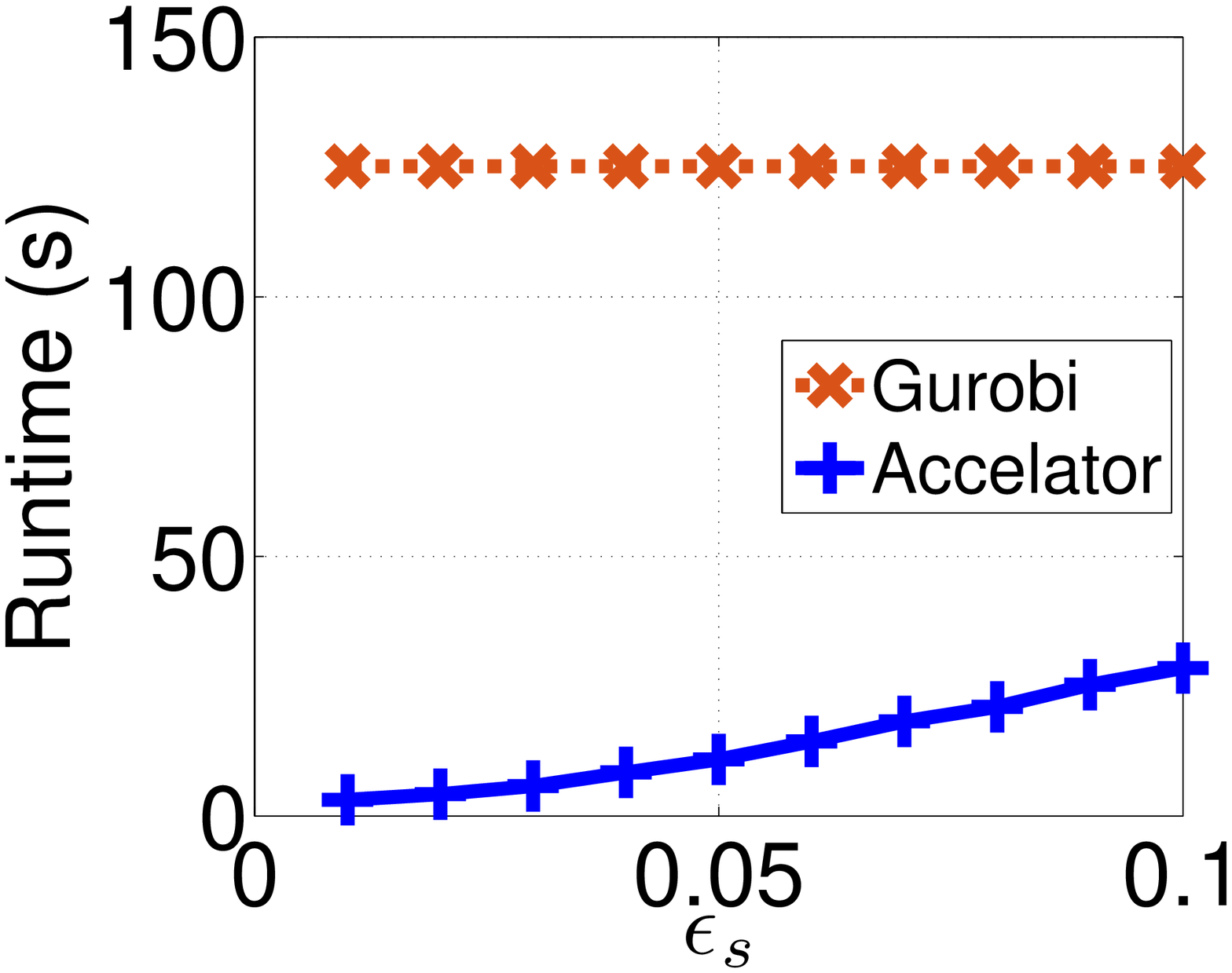}}
    \caption{\emph{Illustration of the relative error and runtime across sample sizes, $\es$, for the real data experiment on the California road network dataset.}}
\label{Figure_RD}
 \end{figure}

\section{Proofs} 
\label{s.proof}

In this section we present the technical lemmas used to prove Theorem \ref{t.main}.  The  approach of the proof is inspired by the techniques in ~\cite{Agrawal}; however the analysis in our case is more involved.  This is due to the fact that our result applies to approximate LP solvers while the techniques in~\cite{Agrawal} only apply to exact solvers. For example, this leads our framework to have three error parameters ($\es, \ef$, $\alpha_d$) while~\cite{Agrawal} has a single error parameter.

The proof has two main steps: (1) show that the solution provided by Algorithm~\ref{alg1} is feasible with high probability (Lemma~\ref{lem:shipra1}); and (2) show that the value of the solution is sufficiently close to optimal with high probability (Lemma~\ref{lem:shipra2}).  
In both cases, we use the following concentration bound, e.g.,~\cite{aw}.
\begin{theorem}[Hoeffding-Bernstein Inequality]\label{thm:hb} Let $u_1, u_2 \ldots, u_s$ be random samples without replacement from the real numbers $r_1, \ldots, r_n$, where $r_j \in [0,1]$. For $t>0$,
	$\Pr\left[ |\textstyle\sum_{j=1}^s u_j - \frac{s}{n}\sum_{j=1}^{n}r_j| \geq t\right]\leq 2\exp\left( \frac{-t^2}{2s\sigma^2_n+t} \right),  $
	where $\sigma^2_n = \frac{1}{n}\sum_{j=1}^n (r_j-\sum_{j=1}^n r_j/n)^2$.
\end{theorem}

\subsection*{Step 1: The solution is feasible}


\begin{lemma}\label{lem:shipra1}
Let $\A$ be a $(\alpha_p, \alpha_d)$-approximation algorithm for packing LPs, $\alpha_p, \alpha_d \geq 1$. For any $\es>0$, $\ef \geq \sqrt{\frac{6(m+2)\log{n}}{\es B}}$,  the solution  Algorithm~\ref{alg1} gives to $\lpa$ is feasible with probability at least $1-1/{2n}$, where the probability is over the choice of samples.
\end{lemma}  

\begin{proof}

Define a \emph{price-realization}, $R(\phi)$, of a price vector $\phi$ as the set $\{r_{ij}=a_{ij}x_j(\phi), j\in[n], i \in [m]\}$ (note that $r_{ij} \in \{0, a_{ij}\}$) and denote, a ``row'' of $R(\phi)$ as
$R_i(\phi)=\{r_{ij}=a_{ij}x_j(\phi), j\in[n]\}$. We say that $R_i(\phi)$ is \emph{infeasible} if $\sum_{j \in [n]} r_{ij} >b_i$. The approach of this proof is to bound the probability that, for a given sample, the sample LP is feasible while there is some $i$ for which $R_i(\phi)$ is not feasible in the original LP.

To begin, note that it naively seems that there are $2^n$ possible realizations of $R(\phi)$, over all possible price vectors $\phi$, as $x_j \in \{0,1\}$. However, a classical result of combinatorial geometry~\cite{cg} shows that there are only $n^m$ possible realizations since each  $R(\phi)$ is characterized by a separation of $n$ points $(\{c_j,a_j\}_{j=1}^n)$ in an $m$-dimensional plane by a hyperplane, where $a_j$ denotes the $j$-th column of $A$.  The maximal number of such hyperplanes is $n^m$. 

Next, we define a sample $S \subset [n]$, $|S|=\es n$ as \emph{$R_i$-good} if  $\sum_{j \in S} r_{ij}  \leq  (1-\ef) \es  b_i$. Let $\xsamp$ be the solution to the sample LP for some sample $S'$. We  say that $S$ (possibly $S \neq 
S'$) is \emph{$\xsamp_i$-good} if $\sum_{j \in S} a_{ij}\xsamp_j  \leq  \toteps  b_i$. The following claim relates these two definitions. 

\begin{claim}\label{claim:ooo}
	If a sample $S$ is $\xsamp_i$-good then $S$ is $R_i$-good.
\end{claim} 
\begin{proof} Denote the dual solution of the sample LP by $\ysamp=[\sphi,\spsi]$. The  dual complementary slackness conditions imply that, if $\spsi_j > 0$ then $\frac{1}{\alpha_d} \leq \xsamp_j \leq 1$ for $j \in [s]$. Further, the allocation rule of Algorithm~\ref{alg1} sets $x_j(\sphi)=1$ if $\sum_{i=1}^{m} a_{ij} \sphi_i <  c_j$, which only occurs when  $\spsi_j>0$.  Therefore, if $x_j(\sphi)=1$, it implies that $\spsi_j>0$, which in turn implies that $\frac{1}{\alpha_d} \leq \xsamp_j$. Consequently, 
	\begin{align*}
	\sum_{j \in S} a_{ij}x_j(\sphi)  
	&\leq \sum_{j \in S} \alpha_d a_{ij}\xsamp_j 
	&\leq  \alpha_d\toteps  b_i \\
	&= (1-\ef) \es b_i, 
	\end{align*}
	which shows that $S$ is $R_i$-good, completing the proof.
\end{proof}


Next, fix the LP and $R(\phi)$. For the purpose of the proof, choose $i \in [n]$ uniformly at random. 
Next, we sample  $\es n$ elements without replacement from $n$ variables taking the values $\{r_{ij}\}$. Call this sample $S$.  Let $X=\sum_{j \in S} r_{ij}$ be the random variable denoting the sum of these random variables. Note that  $\expect{X}=\es\sum_{j \in N}{r_{ij}}$, where the expectation is over the choice of $S$, and  that the events $ \sum_{j \in N}r_{ij} >b_i$ and  $\expect{X}>\es b_i$ are equivalent.  The probability that a sample is $\xsamp_i$-good and $R_i(\phi)$ is infeasible is
\begin{align}
	&\Pr\left[ \sum_{j \in S} a_{ij}\xsamp_j  \leq  \toteps b_i  \wedge  \sum_{j \in N}r_{ij} >b_i\right] \notag\\
	&\leq \Pr\left[ \sum_{j \in S} r_{ij}  \leq  (1-\ef) \es b_i  \wedge  \sum_{j \in N}r_{ij} >b_i\right] \label{ooo}\\
 &\leq \Pr\left[ \sum_{j \in S} r_{ij}  \leq  (1-\ef) \es b_i  \mid  \sum_{j \in N}r_{ij} >b_i\right]\notag\\
	&\leq \Pr\left[ |X - \expect{X}| >
	\ef\es b_i\right]\notag\\
	& \leq 2\exp\left( -\frac{\ef^2 \es^2 b_i^2}{2\es b_i+\ef\es b_i}\right)\label{ooo1}\\
		& = 2\exp\left( -\frac{\ef^2 \es b_i}{2+\ef}\right) \leq \frac{1}{2n^{m+2}},\label{ooo2}
\end{align}
where~\eqref{ooo} is due to Claim~\ref{claim:ooo},~\eqref{ooo1} uses Theorem~\ref{thm:hb}, and~\eqref{ooo2} uses the fact that $B \geq \frac{6(m+2)\log{n}}{\es \ef^2}$.

To complete the proof, we now take a union bound over all possible realizations of $R$, which we bounded earlier by $n^m$, and values of $i$.
\end{proof} 

\subsection*{Step 2: The solution is close to optimal}

To prove that the solution is close to optimal we make two mild, technical assumptions.

\begin{assumption} \label{ass1} For any dual prices $y=[\phi,\psi]$, there are at most $m$ columns such that $\phi^T a_j=c_j$.
\end{assumption}

\begin{assumption}  \label{ass2} Algorithm $\A$ maintains primal and dual solutions $x$ and $y=[\phi, \psi]$ respectively with $\psi>0$ only if $\sum_{j=1}^n a_{ij}x_j <c_j$.
\end{assumption}

Assumption~\ref{ass1} does not always hold; however it can be enforced by perturbing each $c_j$ by a small amount at random (see, e.g.,~\cite{DH09,Agrawal}).  Assumption~\ref{ass2} holds for any ``reasonable'' $(1-\alpha_d)$-approximation dual ascent algorithm, and any algorithm that does not satisfy it can easily be modified to do so.  These assumptions are used only to prove the following claim, which is used in the proof of the lemma that follows. 

\begin{claim}\label{cl:oo}Let $\xsamp$ and $\ysamp=[\sphi, \spsi]$ be solutions of $\A$ to the sampled LP~\eqref{sampleLP2}. Then $\{x_j(\sphi)\}_{j \in [s]}$ and $\{\xsamp_j\}_{j \in [s]}$ differ on at most $m$ values of $j$.
\end{claim}

\begin{proof} 
	For all $j \in [s]$, if $\spsi_j>0$ then $\xsamp_j=1$ by primary complementary slackness, and $x_j(\sphi)=1$  by definition of the allocation rule (recall that if  $\spsi_j>0$ then $\sum_j a_{ij}\sphi_i+\spsi_j=c_j$, by Assumption~\ref{ass2}). Therefore any difference between them must occur for $j$ such that  $\spsi_j=0$.   For all $\spsi_j =0$ such that $\left(c_j  -\sum_j a_{ij}\sphi_i\right) > 0$, it must hold that $\xsamp_j=0$ by complementary slackness, but then also $x_j(\sphi)=0$ by the allocation rule. Assumption \ref{ass1} then completes the proof.
\end{proof}

	

\begin{lemma}\label{lem:shipra2}
Let $\A$ be a $(1, \alpha_d)$-approximation algorithm for packing LPs, $\alpha_d \geq 1$. For any $\es>0$, $\ef \geq \sqrt{\frac{6(m+2)\log{n}}{\es B}}$, the solution Algorithm~\ref{alg1} gives to LP \eqref{LP1} is a $(1-3\ef)/\alpha_d^2$-approximation to the optimal solution with probability at least $1-\frac{1}{2n}$, where the probability is over the choice of samples.  
\end{lemma}

\begin{proof}
Denote the primal and dual solutions to the sampled LP in~\eqref{sampleLP2} of Algorithm~\ref{alg1} by $\xsamp$, $\ysamp=[\sphi,\spsi]$.  For purposes of the proof, we construct the following related LP.  
\begin{align} \label{LP:b}
\maximize \quad & \textstyle \sum_{j=1} ^{n} c_j x_j &  \\
\text{subject to} \quad & \textstyle \sum_{j=1} ^{n} a_{ij} x_j  \le \bnew_i & i \in [m] \notag \\
& x_j \in [0,1]  & j \in [n], \notag
\end{align}
where 
\begin{align*}
\bnew_i = \begin{cases}
\sum_{j=1}^n a_{ij}x_j(\sphi) & \text{if } \sphi_i>0\\
\max\{\sum_{j=1}^n a_{ij}x_j(\sphi),b_i\} &\text{if } \sphi_i=0
\end{cases}
\end{align*}Note that $\bnew$ has been set to guarantee that the LP is always feasible, and that $x(\sphi)$ and $y^*=[\sphi,\psi^*]$ satisfy the (exact) complementary slackness conditions, where $\psi^*_j = c_j - \sum_{i=1}^{m} a_{ij}$ if $x_j(\sphi)= 1$, and $\psi^*_j = 0$ if $x_j(\sphi)\neq 1$.
In particular, note that  $\psi^*$ preserves the exact complementary slackness condition, as $\psi^*_j$ is set to zero when $x_j(\sphi)\neq 1$. Therefore $x(\sphi)$  and $y^*=[\sphi,\psi^*]$ are optimal solutions to LP~\eqref{LP:b}.

A consequence of the approximate dual complementary slackness condition for the solution $\xsamp, \ysamp$ is that the $i$-th primal constraint of LP~\eqref{sampleLP2} is almost tight when $\sphi_i>0$ :  $$\sum_{j \in S} a_{ij}\xsamp_j \geq \frac{(1-\ef)\es}{(\alpha_d)^2} b_i.$$
This allows us to bound $\sum_{j \in S} a_{ij}x_j(\sphi)$ as follows.
$$\sum_{j \in S} a_{ij}x_j(\sphi) \geq \sum_{j \in S} a_{ij}\xsamp_j -m \geq \frac{(1-2\ef)\es}{(\alpha_d)^2} b_i,$$
where the first inequality follows from Claim~\ref{cl:oo} and the second follows from the fact that $B \geq \frac{m(\alpha_d)^2}{\ef \es}$. Thus:
\begin{align*}
&\Pr\left[ \sum_{j \in [s]} r_{ij} \geq \frac{(1-2\ef)\es}{(\alpha_d)^2} b_i \wedge \sum_{j \in [n]}  r_{ij} < \frac{1-3\ef}{(\alpha_d)^2} b_i \right] \\
&\leq \Pr\left[ |X - \expect{X}| >
\frac{\ef\es}{(\alpha_d)^2} b_i\right]\\
& \leq 2\exp\left( -\frac{\ef^2 \es b_i}{2(\alpha_d)^4 +(\alpha_d)^2\ef}\right)
\leq \frac{1}{2n^{m+2}}.
\end{align*}
In the final step, we take $\alpha_d$ close to one, i.e., we assume $3 \geq 2(\alpha_d)^4 +(\alpha_d)^2\ef$. The constant $6$ in the lemma can be adjusted if application for larger $\alpha_d$ is desired. 

Applying the union bound gives that, with probability at least $1-\frac{1}{2n}$,
it holds that $\bnew_i \geq \sum_{j \in [n]} r_{ij} \geq \frac{(1-3\ef)}{(\alpha_d)^2} b_i$.  It follows that, if $x^*$ is an optimal solution to $\lp$, then $\frac{(1-3\ef)}{(\alpha_d)^2}x^*$ is a feasible solution to LP~\eqref{LP:b}.  Thus, the optimal value of LP~\eqref{LP:b} is at least $\frac{(1-3\ef)}{(\alpha_d)^2}\sum_{j=1}^n c_jx^*_j$. 
\end{proof}

\bibliographystyle{aaai} 
\bibliography{_aaai_lp}

\end{document}